\documentclass{amsart}
\usepackage[english,french]{babel}
\usepackage[utf8]{inputenc}
\usepackage[dvips,final]{graphics}
\usepackage{amsmath,amsfonts,amssymb,amsthm,amscd,array,stmaryrd,mathrsfs, mathdots, epigraph}
\usepackage{arydshln}
\usepackage[makeroom]{cancel}
\usepackage{pstricks}
 \usepackage[all]{xy}
 \usepackage{url}
\usepackage{ulem}
\usepackage{multirow, blkarray}
\usepackage{booktabs}
\usepackage{textcomp}
 \usepackage[final]{epsfig}
 \usepackage{color}

\usepackage{hyperref}
\theoremstyle{plain}
\newtheorem{thm}{Theorem}[section]

\newtheorem{lem}[thm]{Lemma}

\newtheorem{prop}[thm]{Proposition}
\newtheorem{defn}[thm]{Definition}

\theoremstyle{definition}
\newtheorem*{rem}{Remark}
\newtheorem*{ex}{Examples}

\begin{document}

\title[Product of a nilpotent and an involutive or idempotent matrix]{Decomposition of a square matrix into a product of a nilpotent matrix and an involutive or idempotent matrix}
\date{}
\author{Flavien Mabilat}
\email{flavien.mabilat@univ-reims.fr}

\maketitle

\selectlanguage{french}
\begin{abstract}

Dans cette note, on va caractériser les matrices carrées de taille $n$ sur un corps commutatif qui peuvent s'écrire comme un produit d'une matrice nilpotente et d'une matrice involutive ainsi que celles qui peuvent s'écrire comme un produit d'une matrice nilpotente et d'une matrice idempotente, à l'aide des invariants de similitude.
\\
\end{abstract}

\selectlanguage{english}

\begin{abstract}

In this note, we give a complete characterization of square matrices of size $n$ over a commutative field which can be expressed as a product of a nilpotent matrix and an involutive matrix and matrices which can be expressed as a product of a nilpotent matrix and an idempotent matrix, in terms of invariant factors.
\\
\end{abstract}
\thispagestyle{empty}


\section{Introduction}

In this note, all fields considered are commutative. Let $\mathbb{K}$ be an arbitrary field. $0_{k,l}$ denotes the zero matrix of $M_{k,l}(\mathbb{K})$ and $I_{n}$ the identity matrix of size $n$. Let $A \in M_{n}(\mathbb{K})$. We denote $\chi_{A}(X):={\rm det}(XI_{n}-A)$ the characteristic polynomial of $A$ (with this definition $\chi_{A}(X)$ is a monic polynomial), $\pi_{A}$ the minimal polynomial of $A$, assumed to be monic. Let $B \in M_{k,l}(\mathbb{K})$, ${}^t B$ is the transpose of $B$. Let $n \geq 1$ and $P(X):=X^{n}+\sum_{i=0}^{n-1} a_{i}X^{i} \in \mathbb{K}[X]$. The companion matrix of $P$ is the matrix $C(P)$ defined as follows:
$C(P):=\begin{pmatrix}
   0 & 0 & \ldots & 0 & -a_{0}  \\
      1 & 0 & \ldots & 0 & -a_{1}   \\
		  0 & \ddots & \ddots & \vdots & \vdots  \\
		  \vdots & \ddots & \ddots & 0 & -a_{n-2}  \\
		  0 & \ldots & 0 & 1 &  -a_{n-1} \\
   \end{pmatrix}$. Let $P$ and $Q$ be two polynomials over $\mathbb{K}$. We will write $P \mid Q$ if $P$ divides $Q$. Let ${\rm deg}(P)$ be the degree of $P$ and ${\rm pgcd}(P,Q)$ be the greatest common divisor of $P$ and $Q$. For all real numbers $x$, $E[x]$ is the integer part of $x$ and $E^{s}[x]$ the ceiling of $x$.
\\
\\
\\\indent Decomposition of square matrices into products of two matrices satisfying remarkable properties is a wide subject of research in linear algebra. Among these properties, one often considers those which verify simple and elegant multiplicative equalities. In particular, many results have been obtained for the most famous of them : nilpotent matrices, involutive matrices and idempotent matrices. More precisely, D. Ž. Djoković has proved that a matrix $A$ is the product of two involutions, that is to say $A=UV$ with $U^{2}=V^{2}=I_{n}$, if and only if $A$ is similar to $A^{-1}$ (see \cite{D}). C. S. Ballantine has characterized the matrices which can be expressed as a product of $k$ idempotents (see \cite{B}), that is to say $A=A_{1}\ldots A_{k}$ with $A_{j}^{2}=A_{j}$, and numerous results about nilpotent factorizations are known (see for instance \cite{Bo,S}).
\\
\\\indent These results all concern decompositions involving matrices satisfying the same property. Hence, a question arises naturally: Can we characterize the matrices that can be represented as a product of two matrices satisfying distinct classical multiplicative properties ? Products of matrices that verify different types of identities have already been studied in certain cases involving invertible matrices (see for instance \cite{SP1,SP2}) but there is still much to be done. In particular, there are few results concerning non-invertible matrices, especially for the matrices cited above (note that the products of a nilpotent and a unipotent matrix have already been studied, see \cite{M}).
\\
\\\indent In this text, we will consider the two following cases: the product of a nilpotent matrix and an involutive matrix and the product of a nilpotent matrix and an idempotent matrix. Note that several results concerning the sum of a nilpotent and an idempotent matrix, called nil-clean matrix, are known (see for instance \cite{Br}). Moreover, numerous elements about products of nilpotent and idempotent elements in a ring have been shown (see for example \cite{CP,Z}).
\\
\\\indent More precisely, we will prove, in section \ref{ninv} and in section \ref{nidem}, the two decomposition results presented below.
	
\begin{thm}
\label{11}

Let $n$ be a positive integer and $\mathbb{K}$ be a commutative field. A matrix $A \in M_{n}(\mathbb{K})$ can be written as a product of a nilpotent matrix and an involutive matrix if and only if ${\rm det}(A)=0$ and $A$ has at most $E\left[\frac{n}{2}\right]$ invariant factors different from both 1 and $X^{j}$, with $1 \leq j \leq n$.

\end{thm}

The proof will be given in the next section, using a theorem from \cite{MZ}. In this same section, we will also recall some results concerning Frobenius normal forms that will be useful to us throughout the remainder of the text.

\begin{thm}
\label{12}

Let $n$ be a positive integer and $\mathbb{K}$ be a commutative field. A matrix $A \in M_{n}(\mathbb{K})$ can be written as a product of a nilpotent matrix and an idempotent matrix if and only if the first nontrivial invariant factor of $A$ has 0 as a root.

\end{thm}

\noindent This theorem will be proved in the last section.

\section{Product of a nilpotent and an involutive matrix}
\label{ninv}

The aim of this section is to prove Theorem \ref{11} and to give some elements about the decomposition studied in this result.

\subsection{Frobenius normal form} 

To fix the notation, we state here the results on the Frobenius decomposition that will be needed in the sequel (for the proofs see for instance \cite{Se} Chapter 6).

\begin{thm}[Theorem of Frobenius]
\label{21}

Let $n$ be a positive integer and $A \in M_{n}(\mathbb{K})$. There exists a unique $n$-tuple of monic polynomials $(P_{1},\ldots,P_{n})$ such that $P_{1} \mid P_{2} \mid \ldots \mid P_{n-1} \mid P_{n}$ and $A$ is similar to the block diagonal matrix $C:=\begin{pmatrix}
   C(P_{1}) &  &  \\
       & \ddots &    \\
		   &   & C(P_{n})
   \end{pmatrix}$, with the convention $C(1):=()$. The polynomials $P_{1},\ldots,P_{n}$ are the invariant factors of $A$ and $C$ is the Frobenius normal form of $A$. In particular, $P_{1}\ldots P_{n}=\chi_{A}$ and $P_{n}=\pi_{A}$.

\end{thm}

\begin{prop}
\label{22}

Two matrices are similar if and only if they have the same invariant factors.

\end{prop}

The invariant factors of a given matrix over a commutative field can be determined using the result below.

\begin{thm}
\label{23}

Let $n$ be a positive integer and $A \in M_{n}(\mathbb{K})$. Let $P_{1},\ldots,P_{n}$ be the invariant factors of $A$. Let $1 \leq k \leq n$, $\Delta_{0}:=1$ and $\Delta_{k}$ be the greatest common divisor (assumed to be monic) of all $k \times k$ minors of the matrix $XI_{n}-A$. We have:
\[\Delta_{k}=P_{1}\ldots P_{k}.\]

\noindent In particular, $P_{k}=\frac{\Delta_{k}}{\Delta_{k-1}}$.

\end{thm}

The following object is not classical. However, it will be very useful for all the results presented in the remaining of this section.

\begin{defn}
\label{24}

Let $n$ be a positive integer and $A \in M_{n}(\mathbb{K})$. Let $P_{1},\ldots,P_{n}$ be the invariant factors of $A$. We denote $i^{*}(A)$ the number of invariant factors of $A$ with at least one nonzero root in the algebraic closure of $\mathbb{K}$, that is to say the number of invariant factors of $A$ different from $1$ and $X^{j}$, with $1 \leq j \leq n$.

\end{defn}

\subsection{$i^{*}$ and products of matrices}

The elements given in this subsection are the key elements for the study of the decomposition \og nilpotent-involution \fg. We begin with a majoration of $i^{*}(AB)$.

\begin{thm}[\cite{MZ} Theorem 1 and \cite{Si} Theorem 4]
\label{25}

Let $n \geq 1$ and $(A,B) \in M_{n}(\mathbb{K})^{2}$. We have: 
\[i^{*}(AB) \geq {\rm max}(0,i^{*}(A)+i^{*}(B)-n).\]

\end{thm}

We also have the following useful result, related to the problem of determining similarity classes $\mathcal{A}$ and $\mathcal{B}$ such that there exist matrices $A \in \mathcal{A}$ and $B \in \mathcal{B}$ whose product $AB$ is nilpotent.

\begin{thm}[\cite{MZ} Theorem 2.2]
\label{26}

Let $n \geq 3$ and $(A,B) \in M_{n}(\mathbb{K})^{2}$ satisfying $i^{*}(A)+i^{*}(B)=n$ and ${\rm det}(A){\rm det}(B)=0$. There exist $A'$ similar to $A$ and $B'$ similar to $B$ such that $A'B'$ is nilpotent.

\end{thm}

Note that the original theorem contains other cases. Here, we state only the case which will be useful in the next subpart.

\subsection{Proof of the decomposition result}

The aim of this subsection is to prove our first theorem. We begin by three easy lemmas.

\begin{lem}
\label{27}

Let $A \in M_{2}(\mathbb{K})$ satisfying ${\rm det}(A)=0$. The matrix $A$ can be written as a product of a nilpotent and an involutive matrix.

\end{lem}

\begin{proof}

Let $u$ be the endomorphism of $\mathbb{K}^{2}$ whose matrix in the canonical basis is $A$. Since ${\rm det}(A)=0$, there exists $x \in \mathbb{K}^{2}$ such that $x \neq 0$ and $u(x)=0$. We complete the family $(x)$ in a basis $(x,y)$ of $\mathbb{K}^{2}$. The matrrix of $u$ in this basis is $B:=\begin{pmatrix}
   0 & a \\
   0 & b   \\
   \end{pmatrix}$. The matrix $A$ is similar to $B$ and to the transpose of $B$. Hence, there exists $P \in GL_{2}(\mathbb{K})$ such that $PAP^{-1}=\begin{pmatrix}
   0 & 0 \\
   a & b   \\
   \end{pmatrix}$. If $b=0$ then $A$ is nilpotent, $A=A \times I_{2}$ and $I_{2}$ is idempotent. If $b \neq 0$, we have the following equality: \[\begin{pmatrix}
   0 & 0 \\
   a & b  
   \end{pmatrix}=\underbrace{\begin{pmatrix}
   0 & 0 \\
   1 & 0  
   \end{pmatrix}}_{N}\underbrace{\begin{pmatrix}
   a & b \\
   (1-a^{2})b^{-1} & -a 
   \end{pmatrix}}_{U}.\]
\noindent Hence, $A=(P^{-1}NP)(P^{-1}UP)$, $P^{-1}NP$ is nilpotent and $P^{-1}UP$ is involutive, since $U^{2}=I_{2}$.

\end{proof}

\begin{lem}
\label{28}

Let $\mathbb{K}$ be a field of odd characteristic and $U \in M_{n}(\mathbb{K})$ be an involutive matrix. There exists $0 \leq r \leq n$ such that $U$ is similar to ${\rm diag}(\underbrace{1,\ldots,1}_{n-r},\underbrace{-1,\ldots,-1}_{r})$. Let $m:={\rm min}(n-r,r)$. We have $m \leq E\left[\frac{n}{2}\right]$and $i^{*}(U)=n-m$. In particular, $i^{*}(U) \geq E^{s}\left[\frac{n}{2}\right]$.

\end{lem}

\begin{proof}

Since $U$ is involutive, $P(U)=0_{n,n}$ with $P(X)=X^{2}-1$. Besides, $P(X)=(X-1)(X+1)$ and $1 \neq -1$ since $\mathbb{K}$ is a field of odd characteristic. Hence, $U$ is diagonalizable and its eigenvalues belong to $\{-1,1\}$, that is to say there exists $0 \leq r \leq n$ such that $U$ is similar to $B:={\rm diag}(\underbrace{1,\ldots,1}_{n-r},\underbrace{-1,\ldots,-1}_{r})$.
\\
\\Suppose that $r \leq n-r$. We have $r \leq E\left[\frac{n}{2}\right]$. We consider the matrix:
\[C:={\rm diag}(\underbrace{C(X-1),\ldots,C(X-1)}_{n-2r},\underbrace{C(P),\ldots,C(P)}_{r}).\]
\noindent We have $P(C)=0_{n,n}$, $-1$ is an eigenvalue of $C$ with multiplicity $r$ and $1$ is an eigenvalue of $C$ with multiplicity $n-r$. Hence, $C$ is similar to $B$ and therefore to $U$. So, the invariant factors of $U$ are $\underbrace{1,\ldots,1}_{r},\underbrace{X-1,\ldots,X-1}_{n-2r},\underbrace{P,\ldots,P}_{r}$ and $i^{*}(U)=n-r \geq n-E\left[\frac{n}{2}\right]=E^{s}\left[\frac{n}{2}\right]$.
\\
\\If $n-r \leq r$, the arguments are similar.

\end{proof}

\begin{lem}
\label{29}

Let $\mathbb{K}$ be a field of characteristic 2 and $U \in M_{n}(\mathbb{K})$ be an involutive matrix. We have $i^{*}(U) \geq E^{s}\left[\frac{n}{2}\right]$.

\end{lem}

\begin{proof}

Let $P_{1},\ldots, P_{n}$ be the invariant factors of $U$. Since the matrix $U$ is involutive, $P(U)=0_{n,n}$ with $P(X)=X^{2}-1$. Besides, $P(X)=(X+1)^{2}$ since $\mathbb{K}$ is a field of characteristic 2. Hence, $\pi_{U}(X)$ divides $(X+1)^{2}$. 
\\
\\If $\pi_{U}(X)=X+1$ then $P_{n}(X)=X+1$. Since $P_{1} \mid \ldots \mid P_{n}$ and ${\rm deg}(P_{1} \ldots P_{n})=n$, we have $P_{1}=\ldots=P_{n}$ and $i^{*}(U)=n \geq E^{s}\left[\frac{n}{2}\right]$.
\\
\\If $\pi_{U}(X)=(X+1)^{2}$ then $P_{n}(X)=(X+1)^{2}$. Since $P_{1} \mid \ldots \mid P_{n}$, we have $P_{i}(X) \in \{1,X+1,(X+1)^{2}\}$. Let $k$ be the number of invariant factors equal to $X+1$ and $j$ be the number of invariant factors equal to $(X+1)^{2}$. We have $i^{*}(U)=k+j$ and $k+2j={\rm deg}(P_{1} \ldots P_{n})=n$. Besides, $1 \leq j \leq E\left[\frac{n}{2}\right]$. Hence, 
\[i^{*}(U)=n-j \geq n-E\left[\frac{n}{2}\right]=E^{s}\left[\frac{n}{2}\right].\]

\end{proof}

\noindent We can now prove the main result of this section.

\begin{thm}

Let $n$ be a positive integer and $\mathbb{K}$ be a commutative field. A matrix $A \in M_{n}(\mathbb{K})$ can be written as a product of a nilpotent matrix and an involutive matrix if and only if ${\rm det}(A)=0$ and $A$ has at most $E\left[\frac{n}{2}\right]$ invariant factors different from both 1 and $X^{j}$, with $1 \leq j \leq n$.

\end{thm}

\begin{proof}

\uwave{$\Rightarrow$ :} Let $n \geq 1$ and $A \in M_{n}(\mathbb{K})$. Suppose that there exists $(N,U) \in (M_{n}(\mathbb{K}))^{2}$ such that $A=NU$, $N$ is nilpotent and $U$ is involutive. 
\\
\\First, we have ${\rm det}(A)={\rm det}(N){\rm det}(U)=0$. 
\\
\\Besides, we have $AU=NUU=N$, that is to say $AU$ is nilpotent. Hence, $i^{*}(AU)=0$, since its invariant factors divide $\chi_{N}(X)=X^{n}$. By Theorem \ref{25} and Lemmas \ref{28} and \ref{29}, we have:
\[0=i^{*}(AU) \geq i^{*}(A)+i^{*}(U)-n \geq i^{*}(A)+E^{s}\left[\frac{n}{2}\right]-n=i^{*}(A)-E\left[\frac{n}{2}\right].\]
So, $i^{*}(A) \leq E\left[\frac{n}{2}\right]$.
\\
\\\noindent \uwave{$\Leftarrow$ :} Let $n \geq 1$ and $A \in M_{n}(\mathbb{K})$. Suppose that ${\rm det}(A)=0$ and $A$ has at most $E\left[\frac{n}{2}\right]$ invariant factors different from both 1 and $X^{j}$.
\\
\\If $n=1$ then $A=(0)=(0)\times (1)$. If $n=2$ then, by Lemma \ref{27}, $A$ can be written as a product of a nilpotent matrix and an involutive matrix. 
\\
\\Suppose now that $n \geq 3$. Let $r:=i^{*}(A)$. By hypothesis, $r \leq E\left[\frac{n}{2}\right]$. Let $P(X):=(X-1)(X+1)$ ($P(X)$ equals to $(X+1)^{2}$ if the characteristic of $\mathbb{K}$ is 2) and $U:={\rm diag}(\underbrace{C(X-1),\ldots,C(X-1)}_{n-2r},\underbrace{C(P),\ldots,C(P)}_{r})$ ($U$ exists since $r \leq E\left[\frac{n}{2}\right]$). The matrix $U$ is involutive and $i^{*}(U)=n-r$. So, we have $n \geq 3$, $A$ non-invertible and $i^{*}(A)+i^{*}(U)=n$. Hence, by Theorem \ref{26}, there exist $A'$ similar to $A$, $U'$ similar to $U$ and $N'$ nilpotent such that $A'U'=N'$. In particular, $A'=N'U'$. Besides, there exists $Q \in GL_{n}(\mathbb{K})$ satisfying $A=QA'Q^{-1}$. Hence, $A=(QN'Q^{-1})(QU'Q^{-1})$, $QN'Q^{-1}$ is nilpotent and $QU'Q^{-1}$ is involutive.

\end{proof}

\begin{ex}
{\rm Let $\mathbb{K}=\mathbb{Z}/5\mathbb{Z}$. 
\begin{itemize}
\item Let $A:=\begin{pmatrix}
   0 & 0 & 0 & 1   \\
   2 & 0 & 1 & 2   \\
	 0 & 4 & 0 & 2  \\
	 3 & 3 & 4 & 1 
   \end{pmatrix}$. The invariant factors of $A$ are $1,1,X+2,X^{2}(X+2)$. Hence, by Theorem \ref{11}, $A$ can be written as a product of a nilpotent matrix and an involutive matrix. 
\item Let $B:=\begin{pmatrix}
    2 & 3 & 4 & 4   \\
    0 & 0 & 4 & 4   \\
		0 & 0 & 2 & 0  \\
		0 & 0 & 0 & 2
   \end{pmatrix}$. The invariant factors of $B$ are $1,X+3,X+3,X(X+3)$. Hence, by Theorem \ref{11}, $B$ can not be written as a product of a nilpotent matrix and an involutive matrix. 
\end{itemize}
}
\end{ex}

\subsection{Some results to conclude}

We give in this subpart some additional information about the decomposition studied in this section.

\begin{prop}
\label{210}

Let $n \geq 1$ and $A \in M_{n}(\mathbb{K})$. If the first nontrivial invariant factor of $A$ has 0 as a root then $A$ can be written as a product of a nilpotent matrix and an involutive matrix.

\end{prop}

\begin{proof}

We give two proofs of this result. Let $n \geq 1$ and $A \in M_{n}(\mathbb{K})$ such that the first nontrivial invariant factor of $A$ has 0 as a root.
\\
\\\uwave{First proof:} We use Theorem \ref{11}.
\\
\\The invariant factors of $A$ are $\underbrace{1,\ldots,1}_{r},X^{i}Q_{r+1}(X),X^{i}Q_{r+2}(X),\ldots,X^{i}Q_{n}(X)$, with $Q_{r+1} \mid \ldots \mid Q_{n}$. Hence, the first invariant factor different from both 1 and a power of $X$ (if it exists) has degree at least two. So, $i^{*}(A) \leq E\left[\frac{n}{2}\right]$. By Theorem \ref{11}, $A$ can be written as a product of a nilpotent matrix and an involutive matrix.
\\
\\\uwave{Second proof:} We give a constructive proof.
\\
\\Let $P(X):=X^{k}+\sum_{i=0}^{k-1} a_{i}X^{i} \in \mathbb{K}[X]$ with $a_{0}=0$. We have:
\[C(P)=\begin{pmatrix}
   0 & 0 & \ldots & 0 & 0  \\
      1 & 0 & \ldots & 0 & -a_{1}   \\
		  0 & \ddots & \ddots & \vdots & \vdots  \\
		  \vdots & \ddots & \ddots & 0 & -a_{k-2}  \\
		  0 & \ldots & 0 & 1 &  -a_{k-1} \\
   \end{pmatrix}=\underbrace{\begin{pmatrix}
   0 & 0 & \ldots & 0 & 0  \\
      1 & 0 & \ldots & 0 & 0   \\
		  0 & \ddots & \ddots & \vdots & \vdots  \\
		  \vdots & \ddots & \ddots & 0 & 0  \\
		  0 & \ldots & 0 & 1 &  0 \\
   \end{pmatrix}}_{N}\underbrace{\begin{pmatrix}
   1 & 0 & \ldots & 0 & -a_{1}  \\
      0 & 1 & \ldots & 0 & -a_{2}   \\
		  \vdots & \ddots & \ddots & \vdots & \vdots  \\
		  \vdots & \ddots & \ddots & 1 & -a_{k-1}  \\
		  0 & \ldots & \ldots & 0 &  -1 \\
   \end{pmatrix}}_{U}.\]
	
\noindent The matrix $N$ is nilpotent and the matrix $U$ is involutive. The Frobenius normal form of $A$ is a block diagonal matrix in which every block has the form of the above matrix. Besides, a block diagonal matrix whose blocks are all nilpotent is nilpotent and a block diagonal matrix whose blocks are all involutive is involutive. Hence, $A$ can be written as a product of a nilpotent matrix and an involutive matrix.

\end{proof}

\begin{prop}
\label{211}

Let $A \in M_{3}(\mathbb{K})$. The matrix $A$ can be written as a product of a nilpotent matrix and an involutive matrix if and only if the first nontrivial invariant factor of $A$ has 0 as a root.

\end{prop}

\begin{proof}

If the first nontrivial invariant factor of $A$ has 0 as a root then $A$ has the desired decompostion, by the previous proposition. Suppose that $A$ can be written as a product of a nilpotent matrix and an involutive matrix. We consier three cases:
\begin{itemize}
\item Suppose that $A$ has one nontrivial invariant factor. Since $A$ is non-invertible, this factor has 0 as a root.
\item Suppose that $A$ has two nontrivial invariant factors. By Theorem \ref{11}, one at least of this factor is necessarily a power of $X$. We denote it by $Q$. If the other factor is not a power of $X$ then, by divisibility, $Q$ is the first nontrivial factor of $A$.
\item Suppose that $A$ has three nontrivial invariant factors. In this case $A$ is scalar and non-invertible, that is to say $A=0_{n,n}$. Hence, the three invariant factors are equal to $X$.
\end{itemize}

\end{proof}

\begin{rem}
{\rm If $n \geq 4$ the result above is false. For instance, ${\rm diag}(C(X-1),C(X^{2}(X-1)))$ can be written as the desired product, by Theorem \ref{11}, but its first nontrivial invariant factor does not have 0 as a root.
}
\end{rem}

\begin{prop}
\label{212}

Let $n \geq 1$ and $A \in M_{n}(\mathbb{K})$. The matrix $A$ can be written as a product of a nilpotent matrix and an involutive matrix if and only if $A$ can be written as a product of an involutive matrix and a nilpotent matrix.

\end{prop}

\begin{proof}

If $A=NU$, with $N$ nilpotent and $U$ involutive, then $A=U(UNU)=U(UNU^{-1})$. If $A=UN$, with $N$ nilpotent and $U$ involutive, then $A=(UNU)U$.

\end{proof}

\noindent We conclude with the two following remarks:
\begin{itemize}
\item The decomposition $A=NU$ is in general not unique. First, we have $A=(-N)(-U)$. However, in general, there are other possibilities. For instance, we have, with $\mathbb{K}=\mathbb{Z}/5\mathbb{Z}$:
\[\begin{pmatrix}
   0 & 0 \\
   1 & 0
   \end{pmatrix}\begin{pmatrix}
   1 & 1 \\
   0 & -1
   \end{pmatrix}=\begin{pmatrix}
   0 & 0 \\
   1 & 1  
   \end{pmatrix}=\begin{pmatrix}
   0 & 0 \\
   3 & 0
   \end{pmatrix}\begin{pmatrix}
   2 & 2 \\
   1 & 3
   \end{pmatrix}.\]
	\\
	
\item The nilpotent and the involutive matrix do not commute in general. For instance, we have:
\[\begin{pmatrix}
   0 & 0 \\
   1 & 0
   \end{pmatrix}\begin{pmatrix}
   1 & 1 \\
   0 & -1
   \end{pmatrix}=\begin{pmatrix}
   0 & 0 \\
   1 & 1  
   \end{pmatrix},~~~~\begin{pmatrix}
   1 & 1 \\
   0 & -1
   \end{pmatrix}\begin{pmatrix}
   0 & 0 \\
   1 & 0
   \end{pmatrix}=\begin{pmatrix}
   1 & 0 \\
   -1 & 0
   \end{pmatrix}.\]
\end{itemize}

\section{Product of a nilpotent and an idempotent matrix}
\label{nidem}

In this last section, we characterize the second kind of matrix decomposition presented in the introduction.

\subsection{Proof of the decomposition theorem}

The aim of this subpart is to prove our second main theorem. To accomplish this task, we need the following result:

\begin{thm}[Theorem of Thompson, \cite{T} Theorem 6]
\label{31}

Let $n \geq 2$ and $1 \leq k \leq n-1$. Let $\mathbb{K}$ be a field, $C \in M_{n}(\mathbb{K})$, and $A \in M_{n-k}(\mathbb{K})$. Let $P_{1},\ldots,P_{n}$ be the invariant factors of $C$ and $Q_{1},\ldots,Q_{n-k}$ be the invariant factors of $A$. Then $A$ is a principal submatrix of some similarity transform of $C$ if and only if the following conditions holds:
\begin{itemize}
\item $P_{i} \mid Q_{i} \mid P_{i+2k}$ for all $1 \leq i \leq n-k$, with $P_{n+1}=\ldots=P_{n+k}=0$;
\item ${\rm deg}(P_{1}\ldots P_{n})=n$;
\item ${\rm deg}(Q_{1}\ldots Q_{n-k})=n-k$.
\end{itemize}

\end{thm}

\noindent In this text, we will only use the necessary condition. Now, we can prove our theorem.

\begin{thm}

Let $n$ be a positive integer and $\mathbb{K}$ be a commutative field. A matrix $A \in M_{n}(\mathbb{K})$ can be written as a product of a nilpotent matrix and an idempotent matrix if and only if the first nontrivial invarant factor of $A$ has 0 as a root.

\end{thm}

\begin{proof}

Let $n$ be a positive integer and $\mathbb{K}$ be a commutative field. Let $A \in M_{n}(\mathbb{K})$. 
\\
\\\uwave{$\Leftarrow$ :} Suppose that the first nontrivial invarant factor of $A$ has 0 as a root. By the divisibility relation, all the nontrivial invarant factors of $A$ have 0 as a root. Hence, the Frobenius normal form of $A$ is a block diagonal matrix in which every block has the following form:
\[\begin{pmatrix}
   0 & 0 & \ldots & 0 & 0  \\
      1 & 0 & \ldots & 0 & -a_{1}   \\
		  0 & \ddots & \ddots & \vdots & \vdots  \\
		  \vdots & \ddots & \ddots & 0 & -a_{k-2}  \\
		  0 & \ldots & 0 & 1 &  -a_{k-1} \\
   \end{pmatrix}=\underbrace{\begin{pmatrix}
   0 & 0 & \ldots & 0 & 0  \\
      1 & 0 & \ldots & 0 & 0   \\
		  0 & \ddots & \ddots & \vdots & \vdots  \\
		  \vdots & \ddots & \ddots & 0 & 0  \\
		  0 & \ldots & 0 & 1 &  0 \\
   \end{pmatrix}}_{N}\underbrace{\begin{pmatrix}
   1 & 0 & \ldots & 0 & -a_{1}  \\
      0 & 1 & \ldots & 0 & -a_{2}   \\
		  \vdots & \ddots & \ddots & \vdots & \vdots  \\
		  \vdots & \ddots & \ddots & 1 & -a_{k-1}  \\
		  0 & \ldots & \ldots & 0 &  0 \\
   \end{pmatrix}}_{V}.\]
	
\noindent The matrix $N$ is nilpotent and the matrix $V$ is idempotent. Hence, $A$ can be written as a product of a nilpotent matrix and an idempotent matrix.
\\
\\\uwave{$\Rightarrow$ :} Suppose that $A$ can be written as a product of a nilpotent matrix and an idempotent matrix. 
\\
\\There exists $(N,V) \in M_{n}(\mathbb{K})^{2}$ such that $A=NV$, $N$ is nilpotent and $V$ is idempotent. Let $R(X)$ be the polynomial $X(X-1)$. We have $R(V)=0_{n,n}$. Hence, $V$ is diagonalizable and its eigenvalues belong to $\{0,1\}$. So, there exist $0 \leq k \leq n$ and $P \in GL_{n}(\mathbb{K})$ such that $PVP^{-1}={\rm diag}(\underbrace{0,\ldots,0}_{k},\underbrace{1,\ldots,1}_{n-k})$. Let $PNP^{-1}:=\begin{pmatrix}
   B & C  \\
   D & E 
   \end{pmatrix}$, with $B \in M_{k}(\mathbb{K})$, $C \in M_{k,n-k}(\mathbb{K})$, $D \in M_{n-k,k}(\mathbb{K})$ and $E \in M_{n-k}(\mathbb{K})$. Let $P_{1},\ldots,P_{n}$ be the invariant factors of $N$ and $Q_{1},\ldots,Q_{n-k}$ be the invariant factors of $E$. We have:
	\[M:=PAP^{-1}=(PNP^{-1})(PVP^{-1})=\begin{pmatrix}
   B & C  \\
   D & E 
   \end{pmatrix}\begin{pmatrix}
   0_{k,k}  & 0_{k,n-k}  \\
   0_{n-k,k} & I_{n-k}
   \end{pmatrix}=\begin{pmatrix}
   0_{k,k} & C  \\
   0_{n-k,k} & E
   \end{pmatrix}.\]
	
\noindent The matrices $M$ and $A$ have the same invariant factors. Besides, we have: $XI_{n}-M=\begin{pmatrix}
   XI_{k} & -C  \\
   0_{n-k,k} & XI_{n-k}-E
   \end{pmatrix}$. Let $1 \leq l \leq n$ and $\Delta_{l}$ be the greatest common divisor of the $l \times l$ minors of $XI_{n}-M$. We consider three cases.
\\
\\i) Suppose that $1 \leq l \leq k$. We can choose the rows and the columns indexed by $1,\ldots,l$. Since $l \leq k$, the minor choosen is $\left|
\begin{array}{ccc}
X & & \\[4pt]
  & \ddots & \\[4pt]
  & & X
\end{array}
\right|=X^{l}$. Hence, $\Delta_{l}$ is a power of $X$ or is equal to 1.
\\
\\ii) Suppose that $n-k+1 \leq l \leq n$. We choose $l$ row indices $1 \leq i_{1} < \ldots < i_{l} \leq n$ and $l$ column indices $1 \leq j_{1} < \ldots < j_{l} \leq n$. Since $l>n-k$ and $\{k+1,\ldots, n\}$ contains $n-k$ elements, at least one of the indices $j_{1},\ldots,j_{l}$ belongs to $\{1,\ldots,k\}$. In particular, $j_{1} \in \{1,\ldots,k\}$. Let $h:={\rm max}(u,~1 \leq u \leq l~{\rm and}~j_{u} \in \{1,\ldots,k\})$. We have two possibilities:
\begin{itemize}
\item Suppose that there exists $1 \leq u \leq h$ such that $j_{u} \notin \{i_{1},\ldots,i_{l}\}$. In this case the minor choosen contains a zero column and therefore is equal to 0.
\item Suppose that $\{j_{1},\ldots,j_{h}\} \subset \{i_{1},\ldots,i_{l}\}$. The minor choosen has the following form $d:=\left|
\begin{array}{cc}
XI_{h} & C'  \\[4pt]
0_{l-h,h} & E'
\end{array}
\right|$, with $E' \in M_{l-h}(\mathbb{K}[X])$. Since the matrix is block triangular, $d=X^{h}{\rm det}(E')=X^{h}T(X)$, with $T \in \mathbb{K}[X]$ satisfying $-\infty \leq {\rm deg}(T) \leq l-h$. 
\end{itemize}

\noindent The $l \times l$ minors are equal to zero or are non constant polynomial with a zero root. Besides, if we choose $\{i_{1},\ldots,i_{l}\}=\{j_{1},\ldots,j_{l}\}=\{1,\ldots,l\}$, we obtain a polynomial of degree $l$ having 0 as a root. Hence, $X$ divides $\Delta_{l}$.
\\
\\iii) Suppose that $k+1 \leq l \leq n-k$. Note that this case implies $k \leq E\left[\frac{n-1}{2}\right]$. Let $I:=\{i_{1},\ldots,i_{l-k}\}$ with $k+1 \leq i_{1} < \ldots < i_{l-k} \leq n$ and $J:=\{j_{1},\ldots,j_{l-k}\}$ with $k+1 \leq j_{1} < \ldots < j_{l-k} \leq n$. Let $E_{I,J}$ be the submatrix extracted from $XI_{n-k}-E$ with theses choices of indices. We choose the rows indexed by $1,\ldots,k,i_{1},\ldots,i_{l-k}$ and the columns indexed by $1,\ldots,k,j_{1},\ldots,j_{l-k}$. The minor of $XI_{n}-M$ obtained is $\left|
\begin{array}{cc}
XI_{k} & C''  \\[4pt]
0_{l-k,k} & E_{I,J}
\end{array}
\right|=X^{k}{\rm det}(E_{I,J})$. Let $\Delta^{'}_{l}$ be the greatest common divisor of all the minors obtained by considering all possible choices of $I$ and $J$. The polynomial $\Delta_{l}$ divides $\Delta^{'}_{l}$. Besides, 
\[\Delta^{'}_{l}={\rm pgcd}(X^{k}{\rm det}(E_{I,J}),~I,J)=X^{k}{\rm pgcd}({\rm det}(E_{I,J}),~I,J).\]
By Theorem \ref{23}, ${\rm pgcd}({\rm det}(E_{I,J}),~I,J)=Q_{1}(X)\ldots Q_{l-k}(X)$. Moreover, since $E$ is a principal submatrix of $PNP^{-1}$, we have, by Thompson's Theorem, $Q_{u} \mid P_{u+2k}$ for all $1 \leq u \leq l-k$. Since $N$ is nilpotent and $u+2k \leq l+k \leq n$, $P_{u+2k}$ is equal to 1 or is a power of $X$. Hence, $Q_{1}(X)\ldots Q_{l-k}(X)$ is equal to 1 or is a power of $X$. So, $\Delta^{'}_{l}$ is a power of $X$. Hence, $\Delta_{l}$ is equal to 1 or is a power of $X$.
\\
\\In the three cases, $\Delta_{l}$ is equal to 1 or is a non constant polynomial with a zero root. Let $v$ be the index of the first nontrivial invariant factor of $A$. By Theorem \ref{23}, $\Delta_{v}$ is the first nontrivial invariant factor of $A$. Hence, the first nontrivial invariant factor of $A$ has 0 as a root.

\end{proof}

\begin{ex}
{\rm Let $\mathbb{K}=\mathbb{Z}/5\mathbb{Z}$. Let $A:=\begin{pmatrix}
     2 & 2 & 1 & 0   \\
     2 & 0 & 1 & 4   \\
		 3 & 2 & 4 & 2  \\
		 1 & 1 & 3 & 0 
   \end{pmatrix}$ and $B:=\begin{pmatrix}
     2 & 3 & 4 & 3   \\
     2 & 2 & 4 & 0   \\
		 0 & 2 & 0 & 1  \\
		 4 & 1 & 3 & 4 
   \end{pmatrix}$.
\begin{itemize}
\item The invariant factors of $A$ are $1,1,X(X+2),X(X+2)$. Hence, by Theorem \ref{12}, $A$ can be written as a product of a nilpotent matrix and an idempotent matrix. 
\item The invariant factors of $B$ are $1,1,X+1,X^{2}(X+1)$. Hence, by Theorem \ref{12}, $B$ can not be written as a product of a nilpotent matrix and an idempotent matrix. 
\end{itemize}
}
\end{ex}

\subsection{Some concluding remarks}

\begin{prop}
\label{32}

Let $n \geq 1$ and $A \in M_{n}(\mathbb{K})$. The matrix $A$ can be written as a product of a nilpotent matrix and an idempotent matrix if and only if $A$ can be written as a product of an idempotent matrix and a nilpotent matrix.

\end{prop}

\begin{proof}

Suppose that $A=NV$, with $N$ nilpotent and $V$ idempotent. We have ${}^t A={}^t V{}^t N$, ${}^t V$ is idempotent and ${}^t N$ is nilpotent. Since ${}^t A$ is similar to $A$, $A$ can be written as a product of an idempotent matrix and a nilpotent matrix. The proof of the reverse part is similar.

\end{proof}

\noindent We conclude with the two following remarks:
\begin{itemize}
\item The decomposition $A=NV$ is in general not unique. For instance, we have, with $\mathbb{K}=\mathbb{Z}/5\mathbb{Z}$:
\[\begin{pmatrix}
   0 & 0 \\
   1 & 0
   \end{pmatrix}\begin{pmatrix}
   1 & 1 \\
   0 & 0
   \end{pmatrix}=\begin{pmatrix}
   0 & 0 \\
   1 & 1  
   \end{pmatrix}=\begin{pmatrix}
   0 & 0 \\
   2 & 0
   \end{pmatrix}\begin{pmatrix}
   3 & 3 \\
   3 & 3
   \end{pmatrix}.\]
	\\
	
\item The nilpotent and the idempotent matrix do not commute in general. For instance, we have:
\[\begin{pmatrix}
   0 & 0 \\
   1 & 0
   \end{pmatrix}\begin{pmatrix}
   1 & 1 \\
   0 & 0
   \end{pmatrix}=\begin{pmatrix}
   0 & 0 \\
   1 & 1  
   \end{pmatrix}~~~~\begin{pmatrix}
   1 & 1 \\
   0 & 0
   \end{pmatrix}\begin{pmatrix}
   0 & 0 \\
   1 & 0
   \end{pmatrix}=\begin{pmatrix}
   1 & 0 \\
   0 & 0
   \end{pmatrix}.\]
\\
\end{itemize}

\end{document}